\documentclass[11pt]{article}

\usepackage{amsrefs}
\usepackage{amsmath,amssymb,amsthm,enumerate,tikz,graphicx}
\usepackage[all,cmtip]{xy}
\usepackage[applemac]{inputenc}

\def \1{{\bf 1}}
\def \aff{\mathrm{aff}}
\def \al{\alpha}
\def \Aut{\operatorname{Aut}}
\def \ax{\operatorname{ax}}
\def \bs{\backslash}
\def \C{{\mathbb C}}

\def \CC{{\cal C}}

\def \CO{{\cal O}}

\def \Cent{\operatorname{Cent}}

\def \det{\operatorname{det}}
\def \df{\ \begin{array}{c} _{\rm def}\\ ^{\displaystyle =}\end{array}\ }

\def \eps{\varepsilon}

\def \F{{\mathbb F}}
\def \fin{\mathrm{fin}}
\def \Ga{\Gamma}
\def \GL{\operatorname{GL}}
\def \ga{\gamma}

\def \hyp{\mathrm{hyp}}

\def \la{\lambda}
\def \La{\Lambda}

\def \M{\mathrm M}
\def \N{{\mathbb N}}

\def \OE{\mathrm{OE}}
\def \ol{\overline}
\def \om{\omega}
\def \op{\mathrm{op}}
\def \Per{\operatorname{Per}}

\def \sgn{\operatorname{sgn}}
\def \SL{\operatorname{SL}}
\def \sm{\smallsetminus}

\def \tors{\mathrm{tors}}
\def \tr{\operatorname{tr}}
\def \Tr{\operatorname{Tr}}

\def \val{\operatorname{val}}

\def \Z{{\mathbb Z}}
\def \({\left(}
\def \){\right)}
\def \[{\left[}
\def \]{\right]}

\newcommand{\drei}[3]{\mbox{\small$\left(\begin{array}{c}#1\\ #2\\ #3\end{array}\right)$}}

\newcommand{\e}
[1]{\emph{#1}\index{#1}}

\newcommand{\mat}
[4]{\(\begin{matrix}#1 & #2 \\ #3 & #4\end{matrix}\)}

\newcommand{\norm}
[1]{\left|\hspace{-1pt}\left|#1\right|\hspace{-1pt}\right|}

\renewcommand{\sp}
[1]{\left\langle #1\right\rangle}

\newcommand{\tto}
[1]{\stackrel{#1}{\longrightarrow}}

\newtheorem{theorem}{Theorem}[section]

\newtheorem{lemma}[theorem]{Lemma}
\newtheorem{corollary}[theorem]{Corollary}
\newtheorem{proposition}[theorem]{Proposition}
\newtheorem{exmple}[theorem]{Example}
\newenvironment{example}[0]{\begin{exmple}\rm}
{\end{exmple}}
\newtheorem{defi}[theorem]{Definition}
\newenvironment{definition}[0]{\begin{defi}\rm}
{\end{defi}}
\newtheorem{remrk}[theorem]{Remark}

\newtheorem{exmples}[theorem]{Examples}
\newenvironment{examples}[0]{\begin{exmples}{\ }\\ 	\vspace{-20pt}\nopagebreak[4]
	\begin{itemize}\rm}{\end{itemize}\end{exmples}\vspace{5pt}}

\begin{document}

\pagestyle{myheadings} \markright{IHARA ZETA FOR NON-UNIFORM TREE-LATTICES}

\title{Tree-lattice zeta functions and class numbers}
\author{Anton Deitmar\thanks{This research was funded by the grant DE 436/10-1 of the Deutsche Forschungsgemeinschaft}\hspace{6pt} \& Ming-Hsuan Kang}

\date{}
\maketitle

{\bf Abstract:}
The theory of Ihara zeta functions is extended to non-compact arithmetic quotients of Bruhat-Tits trees. 
This new zeta function turns out to be a rational function, despite the infinite-dimensional setting.
In general it has zeros and poles, in contrast to the compact case.
The determinant formulas of Bass and Ihara hold true if one defines the determinant as limit of all finite principal minors.
From this analysis, a prime geodesic theorem is derived, which, applied to special arithmetic groups, yields new asymptotic assertions on class numbers of orders in global fields.

\newpage

\tableofcontents

\section*{Introduction}

The Ihara zeta function, introduced by Yasutaka Ihara in \cites{Ihara1,Ihara2} 
is a zeta function counting prime elements in discrete subgroups of rank one $p$-adic groups.
It can be interpreted as a geometric zeta function for the corresponding finite graph, which is a quotient of the Bruhat-Tits building attached to the $p$-adic group \cite{Serre}.
Over time it has been generalized in stages by Sunada, Hashimoto and Bass \cites{Sun1,Sun2,Hash0,Hash1,Hash2,Hash3,Hash4,Bass,Sunada}.
Comparisons with number theoretical zeta functions can be found in the papers of Stark and Terras \cites{ST1,ST2,ST3}.
This zeta function is defined as the product
$$
Z(u)=\prod_p\(1-u^{l(p)}\)^{-1},
$$
where $p$ runs through the set of prime cycles in a finite graph $X$. The product, being infinite in general, converges to a rational function, actually the inverse of a polynomial, and satisfies the famous \e{Ihara determinant formula}
$$
Z(u)^{-1}=\det(1-uA+u^2Q)(1-u^2)^{-\chi},
$$
where $A$ is the adjacency operator of the graph, $Q+1$ is the valency operator and $\chi$ is the Euler number of the graph.
One of the most remarkable features of the Ihara formula is, that in the case of $X=\Ga\bs Y$, where $Y$ is the Bruhat-Tits building of a $p$-adic group $G$ and $\Ga$ is a cocompact arithmetic subgroup of $G$, then the right hand side of the Ihara formula equals the non-trivial part of the Hasse-Weil zeta function of the Shimura curve attached to $\Ga$, thus establishing the only known link between geometric and arithmetic zeta- or L-functions.

In recent years, several authors have asked for a generalization of these zeta functions to infinite graphs.
The paper \cite{Scheja} considers the arithmetic situation, where the graph is the union of a compact part and finitely many cusps. The zeta function is defined by plainly ignoring the cusps, so indeed, it is a zeta function of a finite graph.
In \cite{Clair1} and \cite{Clair2}, the zeta function of a finite graph is generalized to an $L^2$-zeta function where a finite trace on a group von-Neumann algebra is used to define a determinant.
In \cite{Grig}, an infinite graph is approximated by finite ones and the zeta function is defined as a suitable limit.
In \cites{Guido1,Guido2} a relative version of the zeta function is considered on an infinite graph which is acted upon by a group with finite quotient.
In \cite{Chinta} the idea of the Ihara zeta function is extended to infinite graphs by counting not all cycles, but only those which pass through a given point.
In \cite{Lenz} the zeta function is extended to an infinite  graph acted upon by a groupoid and equipped with an invariant measure.

In the present paper we take a different approach by considering cycles which come from geodesics in the universal covering.
This idea goes back to Bass  \cite{Bass} who incorporated torsion in this way.
Surprisingly this yields a convergent Euler product, which extends to a rational function.
This works for  graphs of ``Lie type'', i.e., quotients of Bruhat-Tits buildings by lattices in rank one $p$-adic groups.
These graphs are ``cuspidal'' in the sense that they consist of a compact part and a finite number of cusps. 
In order to derive determinant expressions of the zeta function,
we  introduce the notion of operators of ``determinant class'', which means that the net of all finite principal minors converges, the limit being called the determinant of the operator.
It turns out that with this notion, the classical determinant formulas of Bass and Ihara actually hold without change.
Also, the analysis of the resulting rational function is precise enough to deduce a version of the Prime Geodesic Theorem in this context.
In the case of the Bruhat-Tits tree of $\GL_2$, the Prime Geodesic Theorem can be applied to obtain the following asymptotic result on class numbers:
$$
 \sum_{\La: R(\La)=m}h(\La)
=
\Delta\1_{\Delta\Z}(m)q^m+O\((q-\eps)^m\)
$$
Here the sum ranges over all quadratic orders over the coordinate ring of an affine curve over a finite field of $q$ elements, $h(\La)$ is the class number of the order $\La$ and $R(\La)$ is the regulator.
This result is the function-field analogue of \cite{Sarnak}, see also \cite{classnc}.

Let us now introduce the geometric idea behind this approach.
The classical predecessor of Ihara's zeta function is the zeta function of Selberg, which counts closed geodesics in compact Riemann surfaces.
The latter generalizes to non-compact surfaces as long as they have finite hyperbolic volume.
In this case such a surface is a union of a compact set and finitely many cusp sectors and the typical behavior of a closed geodesic is that it winds around a cusp, going out for a while and then winds back to the compact core.
Below we have drawn this in the case of the quotient of the upper half plane by $\SL_2(\Z)$.
The first picture shows how a geodesic in the universal covering runs through translates of the standard fundamental domain.

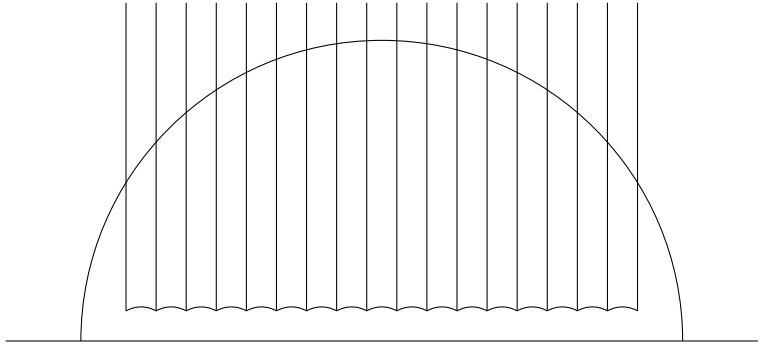
\begin{center}
\begin{tikzpicture}
\draw(-5,0)--(5,0);
\draw(4,0)arc(0:180:4);
\foreach \x in {.2,.6,1,1.4,1.8,2.2,2.6,3,3.4,
 -.2,-.6,-1,-1.4,-1.8,-2.2,-2.6,-3}
{
\draw(\x,.4)--(\x,4.5);
\draw(\x,.4)arc(60:120:.4);
};
\draw(-3.4,.4)--(-3.4,4.5);
\end{tikzpicture}
\end{center}

Instead, one can also leave the fundamental domain fixed and replace the geodesic by the union of its translates, as in the second picture.

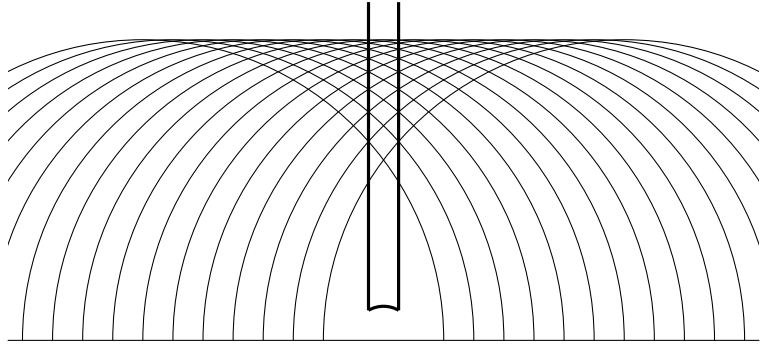
\hspace{-1cm}
\begin{tikzpicture}
\draw(-5,0)--(5,0);
\draw[very thick](-.2,.4)--(-.2,4.5);
\draw[very thick](.2,.4)--(.2,4.5);
\draw[very thick](.2,.4)arc(60:120:.4);
\foreach \x in {.8,1.2,1.6,2,2.4,2.8,3.2,3.6,4,4.4,4.8,5.2,5.6,6,6.4,6.8,7.2}
{
\draw(\x,0)arc(0:180:4);
};
\draw[fill,color=black!00](5,0)--(7.3,0)--(7.3,5)--(5,5);
\draw[fill,color=black!00](-5,0)--(-7.3,0)--(-7.3,5)--(-5,5);
\end{tikzpicture}

Finally, to understand the behavior of the geodesic in the quotient, one only looks at what happens in the fundamental domain, where one now nicely sees, how, after identifying the left and right boundary of the domain, the geodesic winds up and down again.

\hspace{-1cm}
\begin{tikzpicture}
\draw[very thick](-.2,.4)--(-.2,4.5);
\draw[very thick](.2,.4)--(.2,4.5);
\draw[very thick](.2,.4)arc(60:120:.4);
\foreach \x in {.8,1.2,1.6,2,2.4,2.8,3.2,3.6,4,4.4,4.8,5.2,5.6,6,6.4,6.8,7.2}
{
\draw(\x,0)arc(0:180:4);
};
\draw[fill,color=black!00](.2,0)--(7.3,0)--(7.3,5)--(.2,5);
\draw[fill,color=black!00](-.2,0)--(-7.3,0)--(-7.3,5)--(-.2,5);
\draw(-4,0)--(4,0);
\end{tikzpicture}

The analogue of the upper half plane in the $p$-adic setting is the Bruhat-Tits tree $Y$ of a rank one group, together with a lattice $\Ga$ acting on the tree.
In this setting, a cusp sector in the quotient $X=\Ga\bs Y$ is an infinite ray emanating from a compact core.
The following picture shows an example of a graph with 4 cusps, the compact core is not drawn.

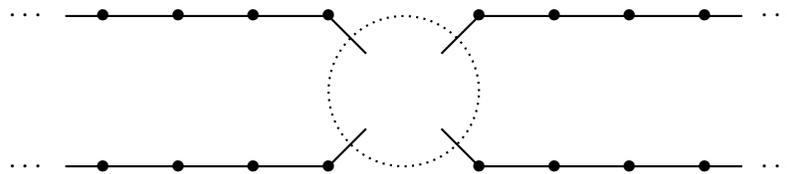
\begin{center}
\begin{tikzpicture}
\draw[thick,dotted](0,0)circle(1);
\draw(1,1)node{$\bullet$};
\draw(1,-1)node{$\bullet$};
\draw(-1,1)node{$\bullet$};
\draw(-1,-1)node{$\bullet$};
\draw[thick](.5,.5)--(1,1)--(4.5,1);
\foreach \x in {2,3,4}
{\draw(\x,1)node{$\bullet$};}
\draw(5,1)node{$\cdots$};
\draw[thick](.5,-.5)--(1,-1)--(4.5,-1);
\foreach \x in {2,3,4}
{\draw(\x,-1)node{$\bullet$};}
\draw(5,-1)node{$\cdots$};
\draw[thick](-.5,.5)--(-1,1)--(-4.5,1);
\foreach \x in {-2,-3,-4}
{\draw(\x,1)node{$\bullet$};}
\draw(-5,1)node{$\cdots$};
\draw[thick](-.5,-.5)--(-1,-1)--(-4.5,-1);
\foreach \x in {-2,-3,-4}
{\draw(\x,-1)node{$\bullet$};}
\draw(-5,-1)node{$\cdots$};
\end{tikzpicture}
\end{center}

Now winding up and down along a cusp means that one considers cycles which move out on the ray, then return once and go back to the compact core.
This corresponds nicely to what happens to a ``geodesic'' in the tree when projected down to the quotient graph.
Note that non-compact quotients with cusps can only occur when the group $\Ga$ has torsion, so $\Ga$ is not the fundamental group of the quotient $X$ which is why one cannot consider the quotient alone but has to take the action of $\Ga$ on $Y$ into the picture.
This is the idea behind Hyman Bass's approach to the zeta function \cite{Bass}.

In the case of the Selberg zeta function, in certain arithmetic cases, it is possible to derive class number growth assertions \cites{Sarnak,class,classnc,class4}.
This can, in principle, also be done for Ihara zeta functions \cite{Iharaclass}. In this paper we apply this technique also in the case of an infinite graph, thus obtaining the above-mentioned growth assertions for class numbers of orders in global fields as follows:
Let  $\CC$ be a smooth projective curve with field of constants $k$ of $q$ elements,  let $\infty$ be a closed point of $\CC$ and let $A$ be the coordinate ring of the affine curve $\CC\sm\{\infty\}$.
Then there exist $\Delta\in\N$, $\eps>0$ such that for $m\to\infty$ one has
$$
 \sum_{\La: R(\La)=m}h(\La)
=
\Delta\1_{\Delta\Z}(m)q^m+O\((q-\eps)^m\)
$$
where the sum runs over all quadratic $A$-orders $\La$ and $h(\La)$ is the class number of $\La$.

\section{Cuspidal tree lattices}
A \e{tree lattice} is a group $\Ga$ together with an action on a locally finite tree $Y$ such that all stabilizer groups $\Ga_e$ of edges $e$ are finite and such that
$$
\sum_{e\mod\Ga}\frac1{|\Ga_e|}<\infty.
$$
As an additional condition, we always assume that the tree $Y$ be \e{uniform}, i.e., the quotient graph $G\bs Y$ is finite, where $G=\Aut(Y)$ is the automorphism group of the tree $Y$.
The compact-open topology makes $G$ a totally disconnected locally compact group. The action of $\Ga$ on $Y$ defines a group homomorphism $\al:\Ga\to G$ and $\Ga$ is a tree lattice if and only if $\al$ has finite kernel and the image $\al(\Ga)$ is a lattice in the group $G$, i.e., $\al(\Ga)$ is a discrete subgroup such that on $G/\Ga$ there exists a $G$-invariant Radon-measure of finite positive volume. 

By replacing $\Ga$ with a subgroup of index two if necessary, we can assume that $\Ga$ acts orientation preservingly on $Y$. 
We will always assume this, as it simplifies the presentation.
In this way oriented edges of the quotient graph $X=\Ga\bs Y$ will be $\Ga$-orbits of oriented edges on $Y$.

A tree lattice of \e{Lie type} is a lattice in a semi-simple $p$-adic group $H$ of rank one, acting on the Bruhat-Tits building $Y$ of $H$, see \cite{BassLub}.

Let $Y$ be a uniform tree.
A \e{path} in $Y$ is a sequence $p=(e_1,e_2,\dots,e_n)$ of oriented edges such that the end point $t(e_j)$ of $e_j$ is the starting point $o(e_{j+1})$  of $e_{j+1}$ for each $j\in\{1,\dots,n-1\}$.
We say that the path is \e{reduced}, if $e_{j+1}\ne e_j^{-1}$ for every $1\le j\le n-1$, where $e^{-1}$ denotes the reverse of the oriented edge $e$.
A \e{ray} in $Y$ is an infinite reduced path $r=(r_1,r_2,\dots)$.
Two rays $r,s$ are \e{equivalent}, if they join at some point, i.e., if there exist $N\in\N$ and $k\in\Z$ such that $r_{j+k}=s_j$ holds for all $j\ge N$.
An equivalence class of rays is called an \e{end} of $Y$.
The set $\partial Y$ of all ends is called the \e{boundary}, or \e{visibility boundary} of $Y$.
For a given vertex $x_0$ and any point $c\in\partial Y$ there exists a unique ray in $c$ which starts at the point $x_0$.
So the visibility boundary is what you see, when you look around from any point in the tree.

Let $c\in\partial Y$ be a boundary point.
Two vertices $x,y$ of $Y$ lie in the same \e{horocycle} with respect to $c$ if there are rays $r,s\in c$ such that $r_0=x$, $s_0=y$ and an $n\in\N_0$ such that $r_n=s_n$.
So a horocycle is the set of all vertices which have the ``same distance'' to the boundary point $c$.
Let $P_c$ be the stabilizer of the boundary point $c$ in the automorphism group $G=\Aut(Y)$. 
Let $N_c$ be the subgroup of $P_c$ of all elements that stabilize a horocycle (and hence all horocycles) with respect to $c$.

The point $c$ is called a \e{cusp} of the tree lattice $\Ga$, if $\Ga\cap P_c=\Ga\cap N_c$ and there exists a horocycle $H$ on which $\Ga\cap N_c$ acts transitively.
It will then act transitively on every horocycle which is nearer to $c$ than $H$.
In this case, any ray in $c$, or rather its image in $X=\Ga\bs Y$ is called a \e{cusp section} of $\Ga$.
Let $c$ be a ray in $Y$ giving a cusp section in $\Ga$.
Recall the \e{valency} of a vertex $x$ is the number of edges ending at $x$.
The valency of the vertices in the ray $c$ can generally behave irregularly.
We say that $c$, and so the cusp, is \e{periodic}, if the sequence of valencies is eventually periodic, i.e., if the vertices of $c$ are $(x_1,x_2,\dots)$ then $c$ is periodic if there exists $N\in\N$ and $k\in\N$ such that $\val(x_n)=\val(x_{n+k})$ holds for all $n\ge N$, where $\val(x)$ is the valency of the vertex $x$.
The smallest possible such $k$ is called the \e{period} of the cusp.
If $Y$ is the Bruhat-Tits tree of a semi-simple $p$-adic group, then every cusp is periodic of period one or two \cite{Lub2}.

We say that the tree lattice $\Ga$ is \e{cuspidal}, if $X=\Ga\bs Y$ is the union of a finite graph and finitely many periodic cusp sections.
Any tree lattice of Lie-type is cuspidal.

\section{The Bass-Ihara zeta function}
Let $\Ga$ be a cuspidal tree lattice of the uniform tree $Y$ and let $X=\Ga\bs Y$ be the quotient graph.
A path $p=(e_1,\dots,e_n)$ in $X$ is called \e{closed}, if $t(e_n)=o(e_1)$.
Then the shifted path $\tau p=(e_2,\dots,e_n,e_1)$ is closed again and the shift induces an equivalence relation on the set of closed paths, whereby two closed paths are equivalent if one is obtained from the other by finitely many shifts.
A \e{cycle} is an equivalence class of closed paths.
If $c$ is a cycle, any power $c^n$, $n\in\N$, which one obtains by running the same path for $n$ times, is again a cycle and a given cycle $c_0$ is called a \e{prime cycle}, if it is not a power of a shorter one.
For every given cycle, there exists a unique prime cycle $c_0$ such that $c=c_0^m$ for some $m\in\N$.
The number $m=m(c)$ is called the \e{multiplicity} of $c$.
A path $p=(e_1,\dots,e_n)$ is called \e{reduced}, if $e_{j+1}\ne e_j^{-1}$ holds for every $1\le j\le n-1$.
A cycle is called reduced if it consists of reduced paths only.

For an edge $e$ of $Y$ we set
$$
w(e)=|\Ga_{o(e)}e|,
$$
this is the cardinality of the $\Ga_{o(e)}$-orbit of $e$.
For an edge $e$ of $X=\Ga\bs Y$ we write $w(e)=w(\tilde e)$, where $\tilde e$ is any preimage of $e$ in $Y$.
For two consecutive edges $e,e'$, i.e., if $t(e)=o(e')$, we write 
$$
w(e,e')=\begin{cases} w(e')& e'\ne e^{-1},\\ w(e')-1 & e'=e^{-1}.\end{cases}
$$

For a closed path $p=(e_1,\dots, e_n)$ let
$$
w(p)=\prod_{j\mod n}w(e_j,e_{j+1}).
$$

\begin{definition}
The Bass-Ihara zeta function for the tree lattice $\Ga$ is defined to be
$$
Z(u)=\prod_c \(1-w(c)u^{l(c)}\)^{-1},
$$
where the product runs over all prime cycles and $l(c)$ is the length of the cycle $c$.
In \cite{Bass} it is shown that this in general infinite product converges for $|u|$ small to a rational function if $X$ is compact.
We will now show the same in the cuspidal case.
\end{definition}

Note the special case of all stabilizer groups being trivial.
In this case one gets
$$
Z(u)=\prod_c\(1-u^{l(c)}\)^{-1},
$$
where now the product runs over \emph{reduced} prime cycles only.

\begin{theorem}\label{thm2.2}
Suppose that the tree lattice $\Ga$ is cuspidal.
Then the product $Z(u)$ converges for small $|u|$ to a rational function.
\end{theorem}

The proof requires a new form of determinant for operators on infinite dimensional spaces, which is given in the next section.

\section{Determinant class operators}
For a given set $I$ consider the formal complex vector spaces
\begin{align*}
P&=P(I)=\prod_{i\in I}\C i,
&S&=S(I)=\bigoplus_{i\in I}\C i.
\end{align*}
We shall denote elements of these spaces as formal sums $\sum_{i\in I}a_ii$, where in the second case the sums are finite in the sense that $a_i=0$ outside a finite set.
For a linear operator $A:S\to P$ and any finite subset $F\subset I$ we define the operator $A_F$ as the composition
$$
S(F)\hookrightarrow S(I)\tto A P(I)\twoheadrightarrow P(F),
$$
where the first arrow is the natural inclusion and the last is the natural projection.

The system of finite subsets $F$ of $I$ is a directed set when equipped with the partial order by inclusion.
Therefore the map $F\mapsto \det(A_F)$ is a net with values in $\C$.

\begin{definition}
We say that $A$ is of \e{determinant class}, if the limit of the net of all principal minors,
$$
\det(A)\df\lim_F\det(A_F)
$$
exists and is $\ne 0$. 
\end{definition}

Define a pairing $\sp{.,.}:P\times S\to \C$ by
$$
\sp{\sum_{i\in I}a_ii,\sum_{i\in I}b_ii}=\sum_{i\in I}a_ib_i.
$$
A permutation $\sigma:I\to I$ is called \e{finite}, if $\sigma(i)=i$ outside a finite set $F$.
In this case we write $\sgn(\sigma)$ for the sign of the permutation $\sigma|_F$. It does not depend on the choice of $F$.

\begin{lemma}[Computation of the determinant]
\label{lem3.2}
Let $A:S\to P$ be a linear operator such that
\begin{itemize}
\item there is a finite set $F_0\subset I$ with $\sp{Ai,i}\ne 0$ outside $F_0$ and $\displaystyle\sum_{i\in I\sm F_0}\left|\log\sp{Ai,i}\right|<\infty$, and
\item $\displaystyle
\sum_{\sigma\text{ finite}}\left|\prod_{i\in I}\sp{Ai,\sigma i}\right|<\infty$.
\end{itemize}
Then $A$ is of determinant class and we have
$$
\det(A)=\sum_{\sigma\text{ finite}}\sgn(\sigma)\prod_{i\in I}\sp{Ai,\sigma i}.
$$
An operator satisfying the condition of this lemma is said to be of \e{strong determinant class}.
\end{lemma}

\begin{proof}
Suppose that $A$ satisfies the conditions of the lemma.
Then
\begin{align*}
\sum_{\sigma}\sgn(\sigma)\prod_{i\in I}\sp{Ai,\sigma i}
&=\lim_F\underbrace{\(\sum_{\sigma\in\Per(F)}\sgn(\sigma)\prod_{i\in F}\sp{Ai,\sigma i}\)}_{=\det(A_F)}\underbrace{\(\prod_{i\in I\sm F}\sp{Ai,i}\)}_{\to 1}.
\end{align*}
The second factor tends to one, so we get the claim.
\end{proof}

\begin{examples}
\item If $I$ is finite, every operator is of determinant class and this formula gives the usual determinant.
\item If the set $I$ is an orthonormal basis of a Hilbert space $H$ on which $T:H\to H$ is a linear operator, then $T$ gives rise to an operator in the algebraic sense above, again written $T$ and if $T$ is of trace class, then $1-T$ is of determinant class and the determinant coincides with the \e{Fredholm determinant}
$$
\det(1-T)=\sum_{k=0}^\infty (-1)^k\tr\wedge^kT.
$$
\item The operators we consider here can not generally be composed.
There are, however, two classes of operators which allow composition.
Firstly, an operator $A$ as above is called a \e{finite column operator}, if it maps $S(I)$ to $S(I)\subset P(I)$.
The name derives from the fact that the corresponding matrix indeed has finite columns, i.e., for every $i\in I$ the set $\{ j\in I: \sp{A i,j}\ne 0\}$ is finite.

If, on the other hand, for each $j\in I$ the set
$\{ i\in I: \sp{A i,j}\ne 0\}$ is finite, we call $A$ a \e{finite row operator}.
A finite row operator possesses a canonical extension to a linear operator $P\to P$.
\item Let $I=\N$, then any operator is given by an infinite matrix. If this matrix is upper triangular with ones on the diagonal, the operator will be of determinant class. The same holds true for lower triangular matrices. The product of a lower triangular times an upper triangular will, however, not always be of determinant class.
This means that the determinant class is not closed under multiplication.
\end{examples}

For an operator $T:S\to P$ we 
say that $T$ is \e{traceable}, if $\sum_{i\in I}|\sp{Ti,i}|<\infty$.
In the case that $T$ is traceable, we define its \e{trace} by
$$
\Tr(T)=\sum_{i\in I}\sp{Ti,i}.
$$

\begin{lemma} 
Suppose that $A,B$ are operators such that $A$ has finite rows or $B$ has finite columns, so that the product $AB$ exists.
Assume further that $A,B$ and  $AB$ are of determinant class.
Then
$$
\det(AB)=\det(A)\det(B).
$$
\end{lemma}

\begin{proof}
Viewed as infinite matrices, we have
\begin{align*}
\sp{ABj,i}=(AB)_{i,j}=\sum_kA_{i,k}B_{k,j}.
\end{align*} 
Under either condition this sum is finite.
In the rest of the proof we assume that $A$ is a finite row operator, the other case being similar.
For a finite set $E\subset I$ we write $A^E$ for the operator given by $A^E_{i,k}=\delta_{i,k}$ unless $i,k$ are both in $E$, in which latter case we have $A^E_{i,j}=A_{i,j}$.
Then $A^E$ has finite rows and columns and $\det(A^E)$ tends to $\det (A)$ as $E\to I$.
Now if $F\supset E$, then one gets $(A^EB)_F=A_F^EB_F$ and so
\begin{align*}
\det\(A^EB\)=\lim_F\det(A^E_FB_F)=\det(A^E)\det(B).
\end{align*}
The right hand side tends to $\det(A)\det(B)$ as $E\to I$.
The fact that $A$ is of finite rows, implies that for every finite set $F\subset I$ there exists a finite set $E$ with $F\subset E\subset I$ such that $(A^EB)_F=(AB)_F$. This implies that $\det(A^EB)\to\det(AB)$ as $E\to I$ and the lemma is proven.
\end{proof}

\subsection{Connectedness}\label{Secconn}
Let $A:S(I)\to P(I)$ be an operator. 
For $j\in I$ write $Aj=\sum_{i\in I}a_ii$ and let $\tilde A(\{j\})$ denote the subset of all $i\in I$ with $a_i\ne 0$.
For any subset $M\subset I$ set
$$
\tilde A(M)=\bigcup_{j\in M}\tilde A(\{j\}).
$$
A subset $M\subset I$ is called \e{$A$-stable} if $\tilde A(M)\subset M$.
The same applies in the relative situation: If $F\subset I$ is any subset, not necessarily finite, then we say that $F$ is \e{$A$-irreducible}, if $F$ and $\emptyset$ are the only $\tilde{A_F}$-stable subsets of $F$.
Further, $F$ is called \e{$A$-connected}, if whenever $F=M\cup N$ with $A$-stable disjoint sets $M,N$, then $M=\emptyset$ or $N=\emptyset$.

Further, we say that $A$ is \e{connected}, if $I$ is $A$-connected.

\begin{lemma}
If $I$ is $A$-irreducible, then the $A$-connected finite subsets $F\subset I$ are cofinal in the directed set of all finite subsets of $I$.
So in particular, any net $F\mapsto n(F)$ has a subnet obtained by restricting to $A$-connected sets $F$, called the \e{$A$-connected subnet}.
\end{lemma}

\begin{proof}
We have to show that any given finite subset $F\subset I$ is contained in an $A$-connected finite subset $C\subset I$.
So let $j_0\in F$. The set $\bigcup_{n\in\N_0}\tilde A(\{j_0\})$ is $A$-stable in $I$, therefore coincides with $I$.
So for each $j\in C$ there is a chain $j_0,j_1,\dots,j_k=j$ such that $j_{\nu+1}\in\tilde A(\{j_\nu\})$ for every $\nu=0,\dots,k-1$.
The set $K(j)=\{j_0,j_1,\dots,j_k=j\}$ is finite and $A$-connected. Let $C$ be the union of all $F(j)$, where $j$ ranges over $F$.
Then $C$ is $A$-connected and finite and contains $F$.
\end{proof}

\begin{definition}
The operator $A$ is of \e{connected determinant class}, if $I$ is $A$-irreducible and $\lim_F\det(A_F)$ exists where the limit is taken over all $A$-connected finite subsets $F\subset I$.
If this is the case, we still write
$$
\det(A)=\lim_F\det(A_F),
$$
where the limit is taken over connected sets $F$ only.
\end{definition}

\begin{example}
Let $I$ be the vertex set of a graph $X$ and let $A$ be the adjacency operator, i.e.,
$$
Ax=\sum_{x'}x',
$$
where the sum runs over all neighbors $x'$ of $x$.
One has that $I$ is $A$-irreducible if and only if $I$ is $A$-connected if and only if $X$ is connected. For every subset $F\subset I$ we write $X_F$ for the full subgraph of $X$ with vertex set $F$. We then get
$$
X_F\text{ is connected}\quad\Leftrightarrow\quad F\text{ is }A\text{-connected.}
$$
\end{example}

\begin{proposition}\label{prop3.7}
If $I$ is $A$-irreducible and $A$ is of determinant class, then $A$ is of connected determinant class.
There are operators, which are of connected determinant class, but not of determinant class.
\end{proposition}

\begin{proof}
In order to prove the first statement, we need to show that $A$-connected finite sets are cofinal in the set of all finite subsets of $I$.
So we need to show that each finite subset $F\subset I$ there exists an $A$-connected finite subset $C\subset I$ with $F\subset C$.
For this let $X$ be the graph with vertex set $I$, where $i,j\in I$ are connected if  $|\sp{Ai,j}|+|\sp{Aj,i}|$ is non-zero. 
For $F\subset I$ let $X_F$ be the full subgraph with vertex set $F$.
Then $F$ is $A$-connected if and only if the graph $X_F$ is connected.
So the assertion boils down to the fact that in a connected graph each finite set of vertices is contained in a connected finite subgraph.

An example of an operator which is of connected determinant class, but not of determinant class, will be given in Section \ref{secIhara}.
\end{proof}

\section{The zeta function as a determinant}
We now apply the theory of determinant class operators to give a proof of Theorem 
\ref{thm2.2}.
So we assume that $\Ga$ is a cuspidal tree lattice acting on the tree $Y$.
We write $X=\Ga\bs Y$ for the quotient graph.
Let $I=\OE(X)$ be the set of all oriented edges of $X$ and define the  operator $T:S(I)\to S(I)$ by
$$
Te=\sum_{e'}w(e,e')e',
$$
where the sum runs over all edges $e'$ with $o(e')=t(e)$.
Note that $T$ is exactly the pushdown of the operator $\tilde T$ on $S(J)$, where $J=\OE(Y)$ given by
$$
\tilde Te=\sum_{e'}e',
$$
where here the sum runs over all edges $e'\ne e^{-1}$ with $o(e')=t(e)$.
We will make this a bit more precise.
As $T$ is an operator of finite rows and columns, we can as well consider it as $T:P(I)\to P(I)$.
Now $P(I)=P(\OE(X))$ can be identified with
$$
P(\OE(Y))^\Ga,
$$
i.e., the $\Ga$-invariants in $P(J)$, where $J=\OE(Y)$.
Hence any operator on $P(J)$ which commutes with the $\Ga$-action, defines an operator on $P(I)$.
In this way the operator $T$ corresponds to the operator $\tilde T:P(J)\to P(J)$ given above.

\begin{lemma}\label{lem4.1}
For any given $n\in\N$ the operator $T^n$ is traceable and the trace is
$$
\Tr(T^n)=\sum_{c:l(c)=n}l(c_0)w(c),
$$
where the sum runs over all cycles $c$ of length $n$ and $c_0$ is the underlying prime cycle of a given cycle $c$.
\end{lemma}

\begin{proof}
Recall that $X$ is a union of a finite graph $X_\fin$ and a finite number of cusp sections.
The best way of thinking of $T$ is that it sends potentials from an edge to the following edges and therefore an edge $e\in I$ only gives a non-zero contribution to the trace, i.e., $\sp{T^ne,e}\ne 0$, if $e$ lies on some cycle of length $n$.
Now if $e$ lies on a cusp section and its distance to $X_\fin$ is bigger than $n$, then it cannot lie on such a cycle, as potentials on a cusp section, which move inward, i.e., towards $X_\fin$, cannot reverse on the cusp section but have to move all the way to $X_\fin$ before returning.
Therefore the sum $\sum_e\sp{T^ne,e}$ is actually finite and so $T^n$ is traceable. 
The claimed formula is clear.
\end{proof}

With this lemma we compute, formally at first,
\begin{align*}
Z(u)^{-1}&=\prod_{c_0}\(1-w(c_0)u^{l(c_0)}\)=\exp\(\sum_{c_0}\log\(1-w(c_0)u^{l(c_0)}\)\)\\
&=\exp\(-\sum_{c_0}\sum_{n=1}^\infty \frac{w(c_0)^nu^{nl(c_0)}}n\)=\exp\(-\sum_c\frac{w(c)u^{l(c)}}{l(c)}l(c_0)\)\\
&=\exp\(-\sum_{n=1}^\infty \frac{u^n}n\sum_{c: l(c)=n}w(c)l(c_0)\)=\exp\(-\sum_{n=1}^\infty \frac{u^n}n\Tr(T^n)\)
\end{align*}

We say that a sequence $A_n$ of operators \e{converges weakly} to an operator $A:S\to P$, if for all $i,j\in I$ the sequence of complex numbers $\sp{A_ni,j}$ converges to $\sp{Ai,j}$.

\begin{lemma}\label{lem3.4}
There is $\al>0$ such that for $u\in\C$ with $|u|<\al$ the series $-\sum_{n=1}^\infty \frac{u^n}nT^n$ converges weakly to an operator we call $\log(1-uT)$.
This operator is traceable and we have
$$
Z(u)^{-1}=\exp\(\Tr\(\log(1-uT)\)\)
$$
for every $u\in\C$ with $|u|<\al$.
\end{lemma}

\begin{proof}
As $Y$ is uniform, there is an upper bound $M$ to the valency of vertices. It follows that $\sum_{e'}w(e,e')\le M$ for every edge $e$, where the sum runs over all edges $e'$ with $o(e')=t(e)$.
It follows that for any two $i,j\in I$ one has
$$
|\sp{T^ni,j}|\le M^n C_n(i,j),
$$
where $C_n(i,j)$ is the number of paths of length $n$ connecting $i$ and $j$.
Let $r$ denote the number of cusp sections in $X$ and let $s$ be the number of oriented edges in $X_\fin$.
For counting the number of paths connecting any two given edges, it suffices to replace each cusp section with a vertex and one oriented edge going out and one in.
Going out a long stretch on a cusp section then is replaced by iterating the loop.
In that way one sees that
$$
|\sp{T^ni,j}|\le M^n (s+2r)^n,
$$
from which the convergence assertion follows.
The trace assertion really is an assertion of changing the order of summation because the trace is itself a sum over $I$.
For $u>0$ all summands are positive, so there is no problem with this interchange of order, for general $u$ one uses absolute convergence, i.e., a Fubini argument, to reach the same conclusion.
\end{proof}

\begin{theorem}\label{lem3.5}
For $|u|$ small enough, the operator $1-uT$ is of strong determinant class and one has
$$
Z(u)^{-1}=\det(1-uT).
$$
This is a rational function of $u$.
\end{theorem}

\begin{proof}
We consider one cusp at a time and for simplicity assume that the period is one.
Then all vertices along the cups can be assumed to have the same valency, say $q+1$.
The modifications for the general case are easy.
Let $(e_0,e_1,e_2,\dots)$ be a ray which represents the cusp.
Write $f_j=e^{-1}$ then we get
\begin{align*}
Te_j&=e_{j+1}+(q-1)f_j,&j\ge 0,\\
Tf_j&=qf_{j-1},&j\ge 1.
\end{align*}
We find that the operator $1-uT$ is represented by the matrix
$$
\begin{array}{c|c|cccccc}
\ &\ &e_0&f_0&e_1&f_1&e_2&f_2\\
\hline
&A&&\al\\
\hline
e_0&\beta&1\\
f_0&&a&1&0&a+b\\
e_1&&b&&1\\
f_1&&&&a&1&&a+b\\
e_2&&&&b&&1\\
f_2&&&&&&a&1
\end{array}
$$
where $a=-u(q-1)$ and $b=-u$. Further the operator $A$ represents what is going on outside the current cusp,
 $\al$ is a finite column vector and $\beta$ a finite row vector.
Note that if one chooses to go up by one unit in the cusp, meaning that $e_j$ will be replaced by $e_{j+1}$ and $f_j$ by $f_{j+1}$, the vectors $\al$ and $\beta$ are disjoint in the sense that there cannot be a permutation $\sigma$ which gives a non-zero contribution to the determinant such that $\sigma(e_0)=f_0$.
It follows that a finite permutation $\sigma$ which gives a non-zero contribution to the determinant must satisfy $\sigma(e_0)=e_0$ or $\sigma(e_0)=e_1$.
In the first case it follows that $\sigma$ is the identity on the whole cusp section.
So, if $\sigma$ is not the identity on the cusp section, the factor $\sp{Te_0,\sigma e_0}$ gives a factor $b=u(q-1)$.
We find that there are not very many choices for such a $\sigma$ and that they come with growing powers of $u$, which for $|u|$ small, force in convergence of the determinant series.
Therefore $1-uT$ is of strong determinant class.
We  compute
\begin{align*}
\det(1-uT)&=\lim_F\det(1-uT_F)\\
&= \lim_F\exp\(-\sum_n\frac{u^n}n\tr T_F^n\)\\
&= \exp\(-\lim_F\sum_n\frac{u^n}n\tr T_F^n\),
\end{align*}
where we have used the continuity of the exponential function.
Next the limit can be interchanged with the sum for small $|u|$ by using dominated convergence by means of a crude estimate of $\tr T_F^n$ similar to the proof of Lemma \ref{lem3.4}.
Finally we have $\lim_F\tr T_F^n=\Tr T^n$ by the same estimate, so that we end up
with the claim.
The other cusps are dealt with in the same fashion.

It remains to show that $\det(1-uT)$ is a rational function in $u$.
For this we again look at one cusp only and again we assume it to be of period one.
As in the proof of Theorem \ref{lem3.5} we see that $\det(1-uT)$ is the determinant of the matrix
$$
\begin{array}{c|c|cccccc}
\ &\ &e_0&f_0&e_1&f_1&e_2&f_2\\
\hline
&A&&\al\\
\hline
e_0&\beta&1\\
f_0&&a&1&0&a+b\\
e_1&&b&&1\\
f_1&&&&a&1&&a+b\\
e_2&&&&b&&1\\
f_2&&&&&&a&1
\end{array}
$$
By starting the cusp section one step later, i.e., moving out on the cusp, we can assume that the vectors $\al$ and $\beta$ each have at most one non-zero entry.
Note that by Lemma \ref{lem3.2}, we are able to compute the determinant of $1-uT$ using row and column reduction as finite matrices. 
By first using column reduction on the column $f_0$ and then applying row reduction on the row $e_0$, there exists a sub-matrix $A'$ of $A$ and a number $c\in\C$ such that
the determinant above is equal to $\det(A)$ plus $cu^2\det(A')$ times the determinant of the infinite matrix
$$
\begin{array}{c|ccccccc}
\\
\hline
&a&0&a+b\\
&b&1\\
&&a&1&&a+b\\
&&b&&1\\
&&&&a&1&&a+b\\
&&&&b&&1\\
&&&&&&&\ddots
\end{array}
$$
Let's call this latter determinant $D$.
Using Laplace expansion along the first row we see that $D$ equals $a$ plus $(a+b)$ times the determinant of 
$$
\begin{array}{c|ccccccc}
\\
\hline
&b&1\\
&&a&&a+b\\
&&b&1\\
&&&a&1&&a+b\\
&&&b&&1\\
&&&&&a&1\\
&&&&&b&&\ddots
\end{array}
$$
But this is $b$ times $D$, so we get
$
D=a+(a+b)bD,
$
or 
$$
D=\frac a{1-(a+b)b}=\frac {(1-q)u}{1-qu^2}.
$$
The proof of Theorem \ref{lem3.5} and therefore of Theorem \ref{thm2.2} is finished.
In the case of higher period than one, the reduction pattern used above to compute $D$ will repeat itself only later, but will still lead to a recursion formula giving rationality.
Other cusps are treated in the same way, so the claim follows in full generality.
\end{proof}

\section{The Ihara formula}\label{secIhara}
In \cite{Bass} it is shown that in the case of a uniform tree lattice $\Ga$ acting on a tree $Y$, the zeta function satisfies
$$
Z(u)^{-1}=\frac{\det(1-uA+u^2Q)}{(1-u^2)^{\chi}},
$$
where $A:S(VY)\to S(VY)$ is the adjacency operator of $Y$, i.e.,
$$
Ay=\sum_{y'}y',
$$
where the sum runs over all vertices $y'$ adjacent to $y\in VY$.
Further, $Q$ is the valency operator minus one, i.e.,
$$
Q(y)=(\val(y)-1)y,
$$
where $\val(y)$ is the valency of the vertex $y$.
Finally, $\chi$ is the Euler number of the finite graph $X$.
Here we use the identification of $S(VX)$ with the set of $\Ga$-invariants in $S(VY)$.

\begin{proof}[Proof of Bass's theorem]
For the convenience of the reader, we give a short account of Bass's proof.
Let $C_0$ be the complex vector space of all maps from $VY$ to $\C$ and $C_1$ the complex vector space of all maps from $\OE(Y)$ to $\C$.
We write the elements of $C_0$ as formal sums $\sum_{p\in VY}c_pp$, where $c_p\in\C$ and likewise for $C_1$.
Let $\partial_0,\partial_1:C_1\to C_0$ be the two boundary operators mapping an edge $e$ to its origin and terminus respectively.
Let $J:\OE(Y)\to\OE(Y)$ be the flip or orientation change operator.
Let $\sigma: C_0\to C_1$ be defined as
$$
\sigma(p)=\sum_{e:\partial_0 e=p}e.
$$
All these operators commute with the $\Ga$-action, so they preserve the finite-dimensional subspace $C_0^\Ga\oplus C_1^\Ga$ of $\Ga$-invariants.
For a given $u\in\C$ we consider the following operators on this finite-dimensional space
$$
L=\mat{1-u^2}{u\partial_0-\partial_1}01\qquad M=\mat1{-u\partial_0+\partial_1}{u\sigma}{1-u^2}.
$$
A simple computation shows that
$$
LM=\mat{1-uA+u^2Q}0{u\sigma}{1-u^2}
$$
and
$$
ML=\mat{1-u^2}{0}{(1-u^2)\sigma}{(1-uT)(1-uJ)}.
$$
Note that $\dim C_0^\Ga=|VX|$ and $\dim C_1^\Ga=|\OE(X)|=2|EX|$.
As $\det(ML)=\det(LM)$ we get
$$
\det(1-uA+u^2Q)(1-u^2)^{2|EX|}=(1-u^2)^{|VX|}\det(1-uT)\det(1-uJ).
$$
Now $\det(1-uJ)=(1-u^2)^{|EX|}$, so we get the claim.
\end{proof}

\begin{theorem}\label{thm5.1}
Let $\Ga$ be a cuspidal tree lattice.
For $|u|$ small enough, the operator $1-uA+u^2Q$ is of connected determinant class and one has
$$
Z(u)^{-1}=\frac{\det(1-uA+u^2Q)}{(1-u^2)^{\chi(X_\fin)}},
$$
where $X_\fin$ is the finite part of $X$, i.e., it is $X$ minus the cusp sections.
If $X$ has at least one cusp, then $1-uA+u^2Q$ is not of determinant class, providing the example promised in Proposition \ref{prop3.7}.
\end{theorem}

\begin{proof}
For any finite subset $F$ of $VX$ let $X_F$ be the full finite subgraph with vertex set $F$.
Assume that $X_F$ is connected and contains $X_\fin$.
Let $Y_F$ be the preimage of $X_F$ in $Y$.
Then $Y_F$ equals $Y$ minus a disjoint union of horoballs, so $Y_F$ is a tree, acted upon by $\Ga$ with compact quotient $X_F$.
So Bass's theorem applies to $Y_F$, giving
\begin{align*}
Z_F(u)^{-1}&=\det(1-uT_F)\\
&= \frac{\det(1-uA_F+u^2Q_F)}{(1-u^2)^{\chi(X_F)}}
\end{align*}
If $X_F$ is connected and contains $X_\fin$, then $\chi(X_F)=\chi(X_\fin)$ as cusps do not contribute to the Euler number.
Therefore we conclude that, as $\lim_F\det(1-uT_F)$ exists, the connected limit over $\det(1-uA_F+u^2Q_F)$ also exists, proving all but the last assertion of the theorem.
It remains to show that $1-uA_F+u^2Q_F$ is not of determinant class.
For this let $F$ be large enough that $X_F$ contains $X_{\fin}$. Then each connected component of $X_F$, which does not contain $X_\fin$ contributes a factor $1-u^2$ to the rational function $(1-u^2)^{\chi(X_F)}$.
So one sees that
$$
\frac{\det(1-uA_F+u^2Q_F)}{(1-u^2)^{|\pi_0(X_F)|}}
$$
converges as $F\to I$, where $\pi_0(X_F)$ is the set of connected components of $X_F$. 
As the denominator alone doesn't converge, the enumerator won't either.
\end{proof}

\section{L-functions}
The Bass-Ihara zeta function van be twisted with a finite dimensional unitary representation $\om:\Ga\to\GL(V)$ of the tree lattice $\Ga$.
For better distinction, we will in this section denote oriented edges of $Y$ by $e,e',e_1,e_2,\dots$ and oriented edges of $X$ by $f,f',f_1,f_2,\dots$.
For each $e\in\OE(Y)$ we denote by $V_e$ a copy of the space $V$, so between any $V_e$ and any $V_{e'}$ there is a natural identification $V_e\cong V_{e'}$.
For each $f\in \OE(X)$ we let
$$
V_f=\(\prod_{e\in f}V_e\)^\Ga,
$$
denote the space of $\Ga$-invariants in the product of which $\Ga$ acts by $(\ga v)_e= \om(\ga)v_{\ga^{-1}e}$. 
Recall that $f$ is an edge of $\Ga\bs Y$, so $f$ is a $\Ga$-orbit of edges in $Y$.
The space $V_f$ is finite-dimensional, isomorphic with $V_e^{\Ga_e}$ for any $e\in f$.
If $f,f'$ are consecutive edges in $X$, so $t(f)=o(f')$, then we define a map
$W(f,f'):V_f\to V_{f'}$ by
$$
W(f,f')v_f=\sum_{e\in f}\sum_{\substack{e':e\to e'\\ e'\ne e^{-1}}}v_{e'}.
$$
For a closed path $p=(f_1,\dots,f_n)$ in $X$ we define $W(p):V_{f_1}\to V_{f_1}$ by
$$
W(p)=W(f_n,f_1)\circ \dots\circ W(f_2,f_3)\circ W(f_1,f_2).
$$
Then $W(p)$ does depend on the path $p$, but $\det(1-u^{l(p)}W(p))$ only depends on the cycle of $p$.
Therefore the product
$$
L(\om,u)=\prod_c\det(1-u^{l(c)}W(c))^{-1}
$$
is well defined as a product over all prime cycles in $X$.
On the space $\bigoplus_{f\in \OE(X)}V_f$ we consider the operator
$$
T_\om(v_f)=\sum_{f'}W(f,f')v_f,
$$
where the sum runs over all $f'\in \OE(X)$ with $o(f')=t(f)$.

Similarly, for a vertex $y$ if $Y$ we let $V_y$ denote a copy of $V$ and for $x$ of $X$ we set
$$
V_x=\(\prod_{y\in x}V_y\)^\Ga.
$$
On the space $\bigoplus_{x\in VX}V_x$ we consider the adjacency operator
$$
A_\om(v_x)=\sum_{y\in x}\sum_{y'}v_{y'},
$$
where the first sum runs over all $y$ in the $\Ga$-orbit $x$ and the second is extended over all neighbors $y'$ of $y$ in $Y$. Finally, $v_{y'}$ is the image of $v_x$ under the canonical identification $V_x\cong V\cong V_{y'}$.
Further, let
$$
Q(v_x)=(\val(y)-1)v_x,
$$
where $y$ is any element of $x$ and $\val(y)$ is the valency of the vertex $y$ in the tree $Y$.

\begin{theorem}
Let $\Ga$ be a cuspidal lattice.
For $|u|$ small enough, the operator $1-uT_\om$ is of determinant class and one has
$$
L(\om,u)^{-1}=\det(1-uT_\om).
$$
This is a rational function of $u$.
For $|u|$ small enough, the operator $1-uA+u^2Q$ is of connected determinant class and one has the Ihara formula,
$$
L(\om,u)^{-1}=\frac{\det(1-uA_\om+u^2Q)}{(1-u^2)^{d\chi(X_\fin)}},
$$
where $X_\fin$ is the finite part of $X$, i.e., it is $X$ minus the cusp sections and $d=\dim(V_\om)$.
\end{theorem}

\begin{proof}
Analogous to Lemma \ref{lem4.1} one sees that
$$
\Tr(T^n)=\sum_{c:l(c)=n}l(c_0)\tr\(W(c)\),
$$
where the sum runs over all cycles $c$ of length $n$ and $l(c_0)$ is the prime cycle underlying $c$.
The first identity follows as in Theorem \ref{lem3.5}.
The argument for rationality is analogous to the proof of Theorem \ref{lem3.5} and the Ihara formula follows as in Theorem \ref{thm5.1}.
\end{proof}

\section{An arithmetic example}\label{arithmexamples}
For background material on this section see \cite{Serre}.
Let $q=p^k$ be a prime power and let $\F=\F_q$ be the finite field of $q$ elements.
On the function field $\F_q(t)$ we put the discrete valuation corresponding to the ``point at infinity'':
\begin{align*}
v\(\frac ab\)=\deg(b)-\deg(a),
\tag{polynomial degree}
\end{align*}Let $K=\widehat{\F(t)}$ denote the local field one gets by completing $\F(t)$ and let $\CO\subset K$ be the corresponding complete discrete valuation ring
$$
\CO=\{ x\in K: v(x)\ge 0\}.
$$
Then $\pi=1/t$ is a uniformizer in $\CO$.
We consider the locally compact group $G=\GL_2(K)$.
Its Bruhat-Tits tree $Y$ can be described as follows.
The vertices are homothety classes of $\CO$-lattices in $K^2$.
Two such lattice classes $[L],[L']$ are connected by an edge if and only if the representatives may be chosen in a way that
$$
\pi L\subset L'\subset L
$$
holds.
The graph described in this way is a tree $Y$ which has constant valency $q+1$.
The natural action of $G$ on $\CO$-lattices induces an action of $G$ on the tree $Y$.
The group $\Ga=\GL_2(\F[t])$ is a discrete subgroup of $G$.
In \cite{Serre} the quotient $\Ga\bs Y$ is described as follows.
For each $n\in\N_0$ let
$$
L_n=\CO e_1\oplus \pi^n\CO e_2,
$$
where $e_1,e_2$ is the standard basis of $K^2$.
Write $x_n$ for the vertex given by the class $[L_n]$.
Then $L_0,L_1,\dots$ is a complete set of representatives for $\Ga\bs Y$, the only edges being $(L_n,L_{n+1})$ for $n\ge 0$.
Put  $\Ga_0=\GL_2(\F)$ and for $n\ge 1$,
$$
\Ga_n=\left\{ \mat ab0d : a,d\in \F^\times,\ b\in\F[t],\ \deg(b)\le n\right\}.
$$
Then for each $n\ge 0$, the group $\Ga_n$ is the stabilizer group of $x_n$.
The group $\Ga_0$ acts transitively on the set of edges with origin $x_0$ and for $n\ge 1$ the edge $(L_n,L_{n+1})$ is fixed by $\Ga_n$. Finally, the group $\Ga_n$ acts transitively on the set of edges with origin $x_n$ distinct from $(x_n,x_{n+1})$.
For this see \cite{Serre} Proposition I.1.3.

So the quotient $X=\Ga\bs Y$ is a single ray.
As in the proof of Theorem \ref{lem3.5} we let $a=-(q-1)u$ and $b=-u$ and we see that $\det(1-uT)$ equals the determinant of
$$
\begin{array}{c|cccccccc}
\\
\hline
&1&a+b\\
&a&1&0&a+b\\
&b&&1\\
&&&a&1&&a+b\\
&&&b&&1\\
&&&&&a&1&&a+b\\
&&&&&b&&1\\
&&&&&&&&\ddots
\end{array}
$$
Again as in the theorem, we see that
$$
Z(u)^{-1}=\det(1-uT)=\frac{1-q^2u^2}{1-qu^2}.
$$

\section{The Prime Geodesic Theorem}
Now suppose $X=\Ga\bs Y$ where $\Ga$ is a cuspidal tree lattice.
If we write
$$
u\frac{Z'}{Z}(u)=\sum_{m=1}^\infty N_mu^m,
$$
then from the Euler product we get
$$
N_m=\sum_{c:l(c)=m}w(c)l(c_0),
$$
where the sum runs over all cycles $c$ and $c_0$ denotes the primitive cycle underlying $c$.
Note that this implies that $N_m$ is a non-negative integer for every $m\in\N$. 
On the other hand, using Theorem \ref{thm2.2}, we get $Z(u)$ is a rational function such that 
$$ Z(u) = \frac{\prod_{j=1}^r(1- a_j u)}{\prod_{k=1}^t(1-b_k u)}.
$$
Then 
$$
N_m=\sum_{j=1}^ra_j^m-\sum_{k=1}^t b_k^m,
$$
When $Y$ is $(q+1)$-regular tree and $X$ has $s$ cusps, as shown in the proof of Theorem \ref{lem3.5}, we have $b_j = \pm \sqrt{q}$ and
$$
N_m=\sum_{j=1}^ra_j^m-2sq^{m/2}\1_{2\Z}(m),
$$
where $\1_{2\Z}$ denotes the indicator function of $2\Z$, so the negative summand at the end only occurs for even $m$.

If $\Ga$ is cuspidal and $N\in\N$, fix a numeration $e_{j,1},e_{j,2},\dots$ of the outward edges of each cusp, here we have $s$ cusps and $j=1,\dots s$.
For a given $N\in\N$ denote $X_N$ the finite subgraph obtained from $X$ by cutting off all edges after $e_N$.

For the graph $X=\Ga\bs Y$ we define the formal space
$$
C_1(X)=\bigoplus_{e\in\OE(X))}\C e
$$
where the direct sum runs over the set $\OE(X)$ of all oriented edges of $X$.
We define the operator $T:C_1(X)\to C_1(X)$ by 
$$
Te=\sum_{e':o(e')=t(e)}w(e,e')e'
$$ 

\begin{theorem}
Suppose that $Y$ is $q+1$-regular and $\Ga$ is  cuspidal.
We also write $T$ for the matrix of $T$ with respect to the basis $(e)_e$ of $C_1(X)$ consisting of all oriented edges.
\begin{enumerate}[\rm (a)]
\item If $X$ is a finite graph, then the matrix $T$ has positive entries only and is connected in the sense of Section \ref{Secconn}.
\item We have $\max_{j=1}^r|a_j|=q$.
\item (Prime Geodesic Theorem)
There exists  $\Delta\in\N$ such that for $m\to\infty$ we have
$$
N_m=\Delta\1_{\Delta\Z}(m)q^m+O\((q-\eps)^m\)
$$
for some $\eps>0$
\end{enumerate}
\end{theorem}

\begin{proof}
The claim (a) is clear as the graph $X$ is connected.
For the other assertions, let's first assume that $\Ga$ is uniform, so $X$ is finite.
The matrix $T$ has positive entries only.
On $C_1(X)$ we define the $1$-norm by
$$
\norm v_1=\sum_e|v(e)|,
$$
and we denote by $\norm._\op$ the corresponding operator norm.
We claim that $\norm T_\op\le q$.
For this note that $\norm{T(e)}=q$ holds for every edge $e$, so that for an arbitrary element $v=\sum_e v_ee$ of $C_1(X)$ one has
$$
\norm{Tv}\le\sum_e|v_e|\norm{Te}=q\sum_e|v_e|=q\norm v.
$$
Next we show that $T$ has $q$ as an eigenvalue.
Since $\norm{Te}=q$ for each edge, it follows that every column of the matrix $T$ has sum $q$, or, in other words, one has $T^tv=qv$, where $v=\sum_ee$.
Since the transpose $T^t$ has the same eigenvalues as $T$, the number $q$ is an eigenvalue.
This proves the assertion (a) and (b) then follows from the Perron-Frobenius Theorem, \cite{Gant} 13.2.2, where we note that what we call a connected matrix is called a \e{non-decomposable} matrix in \cite{Gant}.

Now we prove (b) and (c) in the case of an infinite graph $X$.
For simplicity, we still assume that $X$ has only one cusp $(e_0,e_1,\cdots)$ and $f_i = e_i^{-1}$. On the other hand, the argument can be easily to apply to multi-cusps.
Analogous to the above, we define the formal space $C_1(X_N)$ and on it the operator $A_N$ given by 
$$
A_N(e_N)=qf_N,\quad A_N(f_N)=(q-1)f_{N-1}+e_N
$$
and $A_N(e)=T(e)$ otherwise.

\begin{lemma}\label{lem3.3}
\begin{enumerate}[\rm (a)]
\item For all $N,m\in\N$ we have 
$$
\tr A_N^m=\tr T^m+\1_{2\Z}(m)q^{m/2}.
$$
\item $A_N$ has $q$ for an eigenvalue and every eigenvalue $\la$ of $A$ satisfies $|\la|\le q$.
\end{enumerate}
\end{lemma}

\begin{proof}
We have $\tr T^m=\sum_cw(c)l(c_0)$, where the sum runs over all closed cycles of length $m$ in $X$.
Likewise, we have $\tr A_N^m=\sum_cw'(c)l(c_0)$, where the sum runs over all closed cycles of length $m$ in $X_N$.
There is an injective map from the set of all cycles in $X$ to the set of all cycles in $X_N$.
Whenever a cycle in $X$ contains $e_{N+1}$, the stretch outside $X_N$ is replaced by the a number of loops $(f_N,e_N)$ to match up the length.
Every cycle in $X_N$ of length $m$ lies in the image of this map, with the exception of any cycle of the form $(f_N,e_N)^{m/2}$ when $m$ is even.
Counting the contributions of these cycles gives the claim (a).

(b) We claim that $\norm{A_N}_\op\le q$, where $\norm._\op$ is the operator norm with respect to $\norm._1$.
For this note $\norm {A_N e}_1=q$ holds for ever edge $e$.
Therefore, for given $v=\sum_e v_e e\in C_1(X_N)$ we obtain
$$
\norm{A_N v}_1\le\sum_ev_e\norm{A_N e}_1=q\norm{v}_1,
$$
which implies $\norm{A_N}_\op\le q$ and so $|\la|\le q$ for every eigenvalue of $A_N$.
Let $A_N$ also denote the matrix of $A_N$ with respect to the basis of edges, then we note that all entries of this matrix are $\ge 0$ and the sum over each column is exactly $q$.
The eigenvalues of $A_N$ coincide with the eigenvalues of the transpose matrix $A_N^t$ and now the sum over each row equals $q$, or, equivalently,
$$
A_N^t\drei 1\vdots 1=q\drei 1\vdots 1.
$$
So $A_N^t$ has $q$ for an eigenvalue and so does $A_N$.
The lemma is proven.
\end{proof}
As the graph $X_N$ is connected, the matrix $A_N$ is connected, so the Perron-Frobenius Theorem implies that 
$$
\tr(A_N^m)=\1_{\Delta\Z}(m)q^m+O\((q-\eps)^m\)
$$
for some $\eps>0$, so the theorem follows.
\end{proof}

\begin{definition}
Let $Y$ be a tree. We call an element $\ga\in\Aut(Y)$ an \e{hyperbolic} element, if its \e{length}:
$$
l(\ga)=\min\{ d(\ga y,y):y\in |Y|\}
$$
is $>0$, where $y$ runs through the points of a geometric realization of $Y$.
\end{definition}

If $\ga$ is hyperbolic then there exists an infinite line $...,v_{-1},v_0,v_1,...$ in $Y$, called the \e{axis} of $\ga$ and written $\ax(\ga)$, such that $\ga$ is a translation along this line, i.e., $\ga v_j=v_{j+l}$ with $l=l(\ga)$, see \cite{Serre}.

Recall that an element $\sigma\in\Ga$ is called \e{primitive}, if the equation $\sigma=\tau^n$ with $n\in\N$ and $\tau\in\Ga$ implies $n=1$.
Any hyperbolic element $\ga\in\Ga$ is a positive power $\ga=\ga_0^m$, $m\in\N$ of a primitive element $\ga_0$ which is uniquely determined up to torsion. 

\begin{proposition}\label{prop3.8}
Let $\Ga$ be a cuspidal tree lattice on the uniform tree $Y$ and let $N_m$ be defined as above.
Then for $m\in\N$ we have
$$
N_m=\sum_{\substack{[\ga]\subset\Ga_\hyp\\ l(\ga)=m}}\frac{l(\ga_0)}{|\Cent(\ga,\Ga_{\ax(\ga)})|},
$$
where the sum runs over all conjugacy classes $[\ga]$ in $\Ga$ of hyperbolic elements of length $m$, $\Ga_{\ax(\ga)}$ denotes the point-wise stabilizer in $\Ga$ of the axis of $\ga$  and $\Cent(\ga,\Ga_{\ax(\ga)})$ is the centralizer of $\ga$ in $\Ga_{\ax(\ga)}$.
\end{proposition}

\begin{proof}
We have $N_m=\sum_{c:l(c)=m}w(c)l(c_0)$.
Pick a cycle $c=(\ol e_0,\dots,\ol e_n=\ol e_0)$ of length $m$.
We say that the hyperbolic element $\ga\in\Ga$ \e{closes} $c$, if the axis of $\ga$ can be written as $...,e_{-1},e_0,e_1,\dots$ such that $e_0$ maps to $\ol e_0$ and $e_n$ maps to $\ol e_n$ and $\ga e_0=e_n$.

We now show that for a given $c$ there are $w(c)$ many conjugacy classes of $\ga\in\Ga$ closing $c$.

We first construct one.
For this pick an oriented edge $\ol e_0$ in $c$ pointing in the direction of the cycle.
Choose a pre-image $e_0$ of $\ol e_0$ in $Y$.
Next let $\ol e_1$ be the next edge on $c$ following $\ol e_0$ and choose a lift $e_1$ of $\ol e_1$ which follows $e_0$ and is different from the inverse of $e_0$.
Note that the number of possible choices for $e_1$ is $w(e_0,e_1)$.
Repeat this step until you finished the full cycle $c$. Above, in $Y$, you then have a path $(e_0,\dots,e_n)$, where $e_n$ is another pre-image of $\ol e_0$.
Note that, given $e_0$, there are $w(c)$ many different choices for $e_n$.
There then exists $\ga_0\in\Ga$ with $\ga_0 e_0=e_n$.
As $e_0$ and $\ga_0 e_0=e_n$ are in a line and have the same orientation, $\ga_0$ is hyperbolic and $e_0$ and $\ga_0 e_0$ lie on its axis.
So we have found a $\ga_0$ which closes $c$.

For any given $\ga\in\Ga$ closing $c$, fix a pre-image $f_0\in\ax(\ga)$ of $\ol e_0$.
There exists $\tau\in\Ga$ with $\tau f_0=e_0$.
Conjugating by $\tau$ we see that the set of all $\Ga$-conjugacy classes of elements $\ga$ closing $c$ is in bijection with the set of $\Ga_{e_0}$-conjugacy classes of $\Ga_{\hyp,c,e_0}$, where the latter is the set of all $\ga\in \Ga$ closing $c$ and having $e_0$ in their axis.
We now show that for $\ga\in\Ga_{\hyp,c,e_0}$ the stabilizer of $\ga$ in $\Ga_{e_0}$ equals $\Cent(\ga,\Ga_{\ax(\ga)})$.
For if $\tau\in\Ga_{e_0}$ and $\tau\ga\tau^{-1}=\ga$ with $\ga\in\Ga_{\hyp,c,e_0}$, it follows $\tau(\ax(\ga))=\ax(\ga)$ and as $\tau$ fixes $e_0$, it fixes $\ax(\ga)$ point-wise, hence $\tau\in\Ga_{\ax(\ga)}$. As $\tau$ centralizes $\ga$, we conclude $\tau\in\Cent(\ga,\Ga_{\ax(\ga)})$.

The proposition will finally follow if we show that 
$$
|\Ga_{\hyp,c,e_0}|=w(c)|\Ga_{e_0}|.
$$
As mentioned above, the number of possible choices for $e_n$ is $w(c)$ and if two elements $\ga,\ga'\in\Ga_{\hyp,c,e_0}$ have the same $e_n$, i.e., if $\ga e_0=\ga' e_0$, it follows $\ga^{-1}\ga'\in\Ga_{e_0}$, which indeed implies $|\Ga_{\hyp,c,e_0}|=w(c)|\Ga_{e_0}|$.
\end{proof}

\section{Class numbers}
Let $\CC$ be a smooth projective curve over the finite field $k$ and let $K=k(\CC)$ denote the function field.
We assume that $k$ actually is the constant field of $K$, which means that $k$ is algebraically closed in $K$.
The closed points of $\CC$ correspond to the valuations on $K$.
Fix a closed point $\infty$ of $\CC$ and let $\CC^\aff=\CC\sm\{\infty\}$.
We denote the coordinate ring of this affine curve by $A$.
Then $A$ is  a Dedekind domain whose prime ideals correspond to the points of $\CC^\aff$.
Let $K_\infty$ be the completion of $K$ at $\infty$, then $\Ga=\GL_2(A)$ is a lattice in the locally compact group $G=\GL_2(K_\infty)$.
We consider the action of $\Ga$ on the Bruhat-Tits tree $Y$ of $G$.
Then $\Ga$ acts as a cuspidal tree lattice of period one. The number of cusps equals the number $\GL_2(A)\bs{\mathbb P}^2(K)$ and so this number is the class number $h(A)$ of the Dedekind ring $A$.

\begin{lemma}\label{lem9.1}
For every hyperbolic $\ga\in\Ga$ we have
$$
|\Cent(\ga,\Ga_{\ax(\ga)})|=(q-1)^2,
$$
where $q$ is the cardinality of $k$.
\end{lemma}

\begin{proof}
Let $\ga\in\Ga$ be hyperbolic, then $\Ga$ is a split semi-simple element of $G$, so its centralizer in $G$ is a maximal torus $T$ and this maximal torus happens to be split, so $T\cong K_\infty^\times\times K_\infty^\times$.
Also the axis $\ax(\ga)$ equals the apartment attached to the torus $T$, whence the group of $t\in T$ fixing $\ax(\ga)$ point-wise is isomorphic to $k^\times\times k^\times$.
\end{proof}

We write $\M_2(K)$ for the algebra of $K$-valued $2\times 2$ matrices.
Let $\ga\in\Ga$ be a hyperbolic element.
Then its centralizer in $\GL_2(K)$ is a non-split torus, so its centralizer $L_\ga=\M_2(K)_\ga$ in $\M_2(K)$ is a field, a quadratic extension of $K$.
The set $\La_\ga=L_\ga\cap\M_2(A)$ is an $A$-order in $L_\ga$.
Its unit group is
$$
\La_\ga^\times=\sp{\ga_0}\times F_\ga,
$$
where $F_\ga$ is a finite group.
The number $R(\La_\ga)=l(\ga_0)$ is called the \e{regulator} of the order $\La_\ga$.
We get a map $\psi:\ga_0\mapsto\La_{\ga_0}$ from the set of primitive
hyperbolic conjugacy classes in $\Ga$ to the set of 
isomorphy classes of $A$-orders in quadratic extensions of $K$.

\begin{proposition}
The map $\psi$ is surjective and each given $A$-order $\La$ in a quadratic extension $L/K$ has $h(\La)(q-1)^2$ pre-images.
\end{proposition}

\begin{proof}
Let $\La$ be an $A$-order in a quadratic extension $L$ of $K$.
Let $v_1,v_2$ be an $A$-basis of the free $A$-module $\La$.
Then $v_1,v_2$ also is a $K$-basis of $L$ and we get an embedding $\sigma:L\hookrightarrow\M_2(K)$ by sending $y\in L$ to the matrix of the map $x\mapsto xy$ in this basis.
Then $\La=\sigma^{-1}(\M_2(A))$.
Let $\la_0\in\La$ be a generator of the $\La^\times/\La^\times_\tors$ and let $\ga_0=\sigma(\la_0)$.
Then $\sigma(L)=L_{\ga_0}$ and $\sigma(\La)=\La_{\ga_0}$, so surjectivity is established.

For a given embedding $\sigma:L\hookrightarrow\M_2(K)$ we write $\La_\sigma$ for the $A$-order $\sigma^{-1}(\M_2(A))$.
For a given $A$-order $\La\subset L$ we write $\Sigma(\La)$ for the set of all $\sigma: L\hookrightarrow\M_2(K)$ with $\la_\sigma=\La$.
The group $\Ga=\GL_2(A)$ acts on $\Sigma(\La)$ by conjugation.

\begin{lemma}\label{lem4.3}
The quotient $\Sigma(\La)/\Ga$ is finite and has cardinality $h(\La)$.
\end{lemma}

\begin{proof}
The proof is similar to the proof of Lemma 2.4 in \cite{Iharaclass} or Lemma 2.3 in \cite{class}.
\end{proof}
As $\ga_0$ in a given $\La$ is only unique up to multiplication by an element of $\Cent(\ga_0,\Ga_{\ax(\ga_0)})$, we get the proposition from Lemma \ref{lem9.1} and Lemma \ref{lem4.3}.
\end{proof}

We write 
$$
N_m=N_m^P+N_m^R,
$$
where $N_m^P$ is the sum over all primitive conjugacy classes $[\ga]$ with $l(\ga)=m$.

\begin{lemma}
For $m\to\infty$ we have
$$
N_m^P=N_m+O(q^{m/2}).
$$
\end{lemma}

\begin{proof}
The expression $N_m=\sum_{j=1}^ra_j^m-2sq^{m/2}\1_{2\Z}(m)$ together with $|a_j|\le q$ yields
$N_m\le rq^m$. Using this, we estimate
\begin{align*}
N_m^R&=\sum_{\substack{d\mid m\\ d<m}}\sum_{\substack{[\ga_0]\\ l(\ga_0)=d}}\frac{l(\ga_0)}{(q-1)^2}\le\sum_{\substack{d\mid m\\ d<m}}N_d\le \sum_{\substack{d\mid m\\ d<m}}rq^d\le \frac{mr}2q^{m/2}.
\end{align*}
The claim follows.
\end{proof}

We finally can put things together now to get the following class number estimate.

\begin{theorem}
[Class number asymptotics]
Let $\CC$ be a smooth projective curve with field of constants $k$ of $q$ elements, fix a closed point $\infty$ of $\CC$ and let $A$ be the coordinate ring of the affine curve $\CC\sm\{\infty\}$.
Then there exist $\Delta\in\N$, $\eps>0$ such that
$$
 \sum_{\La: R(\La)=m}h(\La)
=
\Delta\1_{\Delta\Z}(m)q^m+O\((q-\eps)^m\)
$$
where the sum runs over all quadratic $A$-orders $\La$ and $h(\La)$ is the class number of $\La$.
\end{theorem}

\begin{proof}
The theorem follows from the prime geodesic theorem, Proposition \ref{prop3.8} and the considerations of this section.
\end{proof}

\begin{corollary}
In the special case of the polynomial ring $A=k[x]$ we get
$$
\sum_{\La: R(\La)=m}h(\La)=2q^m\1_{2\Z}(m)+O(q^{m/2}),
$$
where the sum runs over all quadratic $A$-orders.
\end{corollary}

\begin{proof}
This follows as the theorem together with the explicit computation of the zeta function in this case.
\end{proof}

\section{Nagao rays}
\begin{definition}
Assume that the quotient graph $X=\Ga \backslash Y$ is a ray $(x_0,x_1,x_2,\cdots)$ where $x_i$ are vertices of $X$. Let $e_i=(x_i,x_{i+1})$ be a directed edge and its inverse is denoted by $f_i$. The ray $X$ is called a \e{Nagao ray} if for $n\in\N$ one has
\begin{align*}
w(e_0)&=q_0+1, \\
w(e_n)&=1, \\
w(f_0)&=q_1, \\
w(f_n)&=q_{n+1}.
\end{align*}
Here we assume that all $q_i\in\N$.
Note that the valency of the preimage of $x_i$ in $X$ is equal to $q_i+1$.
\end{definition}

\begin{definition}
Let $P_n$ be the collection of all  closed paths $p$ of length $2n$ and $w(p)\ne 0$. 
Note that the closed paths in $P_n$ do not have backtracking of the form $(f_i,e_i)$ for $i\ge 1$, but they may contain  backtracking of the form $(f_0,e_0)$, called left backtracking, and the backtracking of the form $(e_i,f_i)$ for some $i \geq 0$, called right backtracking.
\end{definition}

\begin{center}
\begin{tikzpicture}
\node (1) at ( 0,0) [circle,draw,label=above:$x_0$] {};
\node (2) at ( 1,0) [circle,draw,label=above:$x_1$] {};
\node (3) at ( 2,0) [circle,draw,label=above:$x_2$] {};
\node (4) at ( 3.2,0) [] {$\cdots$};
\fill (1,-0.5) circle (2pt);
\draw [->] (1.east) -- (2.west);
\draw [->] (2.east) -- (3.west);
\draw [->] (3.east) -- (4.west);
\draw [->] (1,-.5) -- (2,-.5); 
\draw [->] (1,-.5)-- (4,-.5)--(4,-.6)--(0,-.6)--(0,-.7)--(3,-.7)--(3,-.8)--(-0.1,-.8)--(-0.1,-.5)--(0.9,-.5);
\end{tikzpicture}

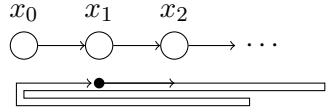
Figure 1: A typical closed path in $P_n$ with two instances of right backtracking and two of left backtracking.
\end{center}

Now given a closed path $p=(p_0,\cdots,p_n=p_0)$ in $P_n$, where the $p_j$ are oriented edges. Let $i$ be the smallest index satisfying $p_i = f_0$ and $j$ be the second smallest index (or the smallest index if there is only one such index) satisfying $(p_j,p_{j+1})=(e_k,f_k)$ for some $k$. We call the number $k$ the right backtrack index.
Then we have the following two injective maps $\rho_1,\rho_2 : P_n \mapsto P_{n+1}$ given by 
$$ \rho_1(p)= (p_0,\cdots,p_i,e_0,f_0,p_{i+1},\cdots,p_n)$$
and 
$$ \rho_2(p)= (p_0,\cdots,p_j,e_{k+1},f_{k+1},p_{j+1},\cdots,p_n)$$
Note that there are unique closed paths $\mathfrak{p}_n$ and $\mathfrak{p}_n'$ in $P_n$ starting from $e_n$ and $f_n$ respectively.

\begin{lemma}
For $n\geq 0$, $P_{n+1}$ is the disjoint union of $\rho_1(P_n)$ and $\rho_2(P_n)$ and $\{\mathfrak{p}_{n+1},\mathfrak{p}_{n+1}'\}.$ Especially,
$$ |P_{1}|=2 \quad \mbox{and} \quad |P_{n+1}| = 2 |P_{n}|+2 \quad\mbox{for all $n>0$}. $$
\end{lemma}
\begin{proof}
This is easy to see.
\end{proof}

\begin{theorem}
Let $X$ be a Nagao ray of Lie type. Then it is periodic of period one or two.
\begin{enumerate}[\quad\rm (a)]
\item If $X$ is of period one, then $q_i=q$ for some $q$ and all $i$.
In this case the Bass-Ihara zeta function is
$$ 
Z(u) =\frac{(1-q u^2)}{(1-q^2 u^2)}.
$$
Moreover, the prime geodesic theorem becomes
$$ 
N_m= 2 q ^m\1_{2\Z}(m)+O(q^{m/2}).
$$
\item If $X$ is of period two, then $q_{2i}=q_0$ and $q_{2i+1}=q_1$ for all $i\ge 0$.
In this case the Bass-Ihara zeta function is
$$ Z(u) =\frac{(1+ q_0 u^2)(1- q_0 q_1 u^2)}{(1- q_0 q_1 u^4)}.$$
Moreover, the prime geodesic theorem is 
$$ N_m= 2 (q_0q_1)^{m/2}\1_{2\Z}(m)+O\bigg((q_0q_1)^{m/4}+ q_0^{m/2}\bigg).$$
\end{enumerate}
\end{theorem}

\begin{proof}
(a) This is in Section \ref{arithmexamples}.

(b)
Now assume that $X$ is of period two such that there are two positive integers $p$ and $q$ and when $i$ is even, $q_i=q_0$; when $i$ is odd, $q_i=q_1$. In this case, 
we split $P_n$ into two parts: $A_n$ and $B_n$, where $A_n$ and $B_n$ contain closed paths with odd and even right backtrack indexes.
Note that when $n=2k$ is even, $\mathfrak{p}_{n},\mathfrak{p}_{n}' \in B_n$ and 
$$ w(\mathfrak{p}_{n})=w(\mathfrak{p}_{n}')=q_0^k q_1^{k} (q_0-1).$$
When $n=2k+1$ is odd, $\mathfrak{p}_{n},\mathfrak{p}_{n}' \in A_n$ and 
$$ w(\mathfrak{p}_{n})=w(\mathfrak{p}_{n}')=q_0^{k+1} q_1^{k}(q_1-1).$$
On the other hand, for any $p \in P_n$, we have $\rho_1(p) \in A_{n+1}$; $\rho_2(p) \in B_{n+1}$ if $p \in A_n$; $\rho_2(p) \in A_{n+1}$ if $p \in B_{n}.$
Denote $\sum_{p \in A_n} w(p)$ by $N_{1,2n}$ and $\sum_{p \in B_n} w(p)$ by $N_{2,2n}$ for short. Then we have $N_{2n}=N_{1,2n}+N_{2,2n}$ for all $n$.

Moreover, When $n=2k$,
\begin{align*}
N_{1,2n+2} &=\sum_{p' \in P_n} w(\rho_1(p'))  + \sum_{p' \in B_n} w(\rho_2(p')) +  2 w(\mathfrak{p}_{n})\\
&=  q_0(q_1-1)N_{2n} + \frac{(q_1-1)q_0}{(q_0-1)}N_{2,2n}  + 2q_0^{k+1} q_1^{k}(q_1-1)
\end{align*}
and 
\begin{align*}
N_{2,2n+2} &=  \sum_{p' \in A_n} w(\rho_2(p'))= \frac{(q_0-1)q_1}{(q_1-1)}N_{1,2n} .
\end{align*}
when $n=2k-1$,
\begin{align*}
N_{1,2n+2} &=\sum_{p' \in P_n} w(\rho_1(p'))  + \sum_{p' \in B_n} w(\rho_2(p'))\\
&= q_0(q_1-1)N_{2n}  + \frac{(q_1-1)q_0}{(q_0-1)} N_{2,2n}
\end{align*}
and 
\begin{align*}
N_{2,2n+2} &=  \sum_{p' \in A_n} w(\rho_2(p')) + 2 w(\mathfrak{p}_{n}) \\
&= \frac{(q_0-1)q_1}{(q_1-1)}N_{1,2n}+2q_0^k q_1^{k} (q_0-1).
\end{align*}
Combining above formula, we can find the recursive formula of $N_n$ as follows.

For $n=2k$,
\begin{align*}
N_{2n+2} &= N_{1,2n+2} + N_{2,2n+2} \\
&= q_0(q_1-1)N_{2n} + \frac{(q_1-1)q_0}{(q_0-1)}N_{2,2n}  +2q_0^{k+1} q_1^{k}(q_1-1) +\\
&\ \ \ \ + \frac{(q_0-1)q_1}{(q_1-1)}N_{1,2n} \\
&= q_0(q_1-1)N_{2n} + q_0 q_1 N_{1,2n-2} + 4q_0^{k+1} q_1^{k}(q_1-1) +\\
&\ \ \ \ +
 q_0 q_1 N_{1,2n-2}+ q_0q_1(q_0-1)N_{2n-2} \\
 &= q_0(q_1-1)N_{2n} +q_0^2 q_1N_{2n-2}+ 4q_0^{k+1} q_1^{k}(q_1-1).
\end{align*}
Similarly, for $n=2k-1$, we have 
\begin{align*}
N_{2n+2} 
 &= q_0(q_1-1)N_{2n} +q_0^2 q_1N_{2n-2}+ 4q_0^{k} q_1^{k}(q_0-1).
\end{align*}
Together with the initial condition $N_0=0$ and $N_2= 2 q_0(q_1-1)$, one can show by induction that 
$$ N_{4k-2} = 2 q_0^{2k-1}q_1^{2k-1}- 2 q_0^{2k-1} \qquad \mbox{and} \quad N_{4k} = 2q_0^{2k}q_1^{2k} + 2q_0^{2k}- 4 q_0^k q_1^k.$$
We conclude that the Bass-Ihara zeta function in this case equals 
\begin{align*}
Z(u) = \exp\(\sum_{n=1}^\infty \frac{N_n}{n}u^n\) =\frac{(1+ q_0 u^2)(1- q_0 q_1 u^2)}{(1- q_0 q_1 u^4)}.
\tag*\qedhere
\end{align*}\end{proof}

{\small Mathematisches Institut, 
Auf der Morgenstelle 10, 
72076 T\"ubingen,
Germany,
\tt deitmar@uni-tuebingen.de}

{\small Department of Mathematics, National Chiao-Tung University, Hsinchu, Taiwan
\tt mhkang@math.nctu.edu.tw}

\end{document}